\documentclass[english]{article}
\usepackage[T1]{fontenc}
\usepackage[latin9]{inputenc}
\usepackage{babel}

\usepackage{amsthm}
\usepackage{amsmath}
\usepackage{graphicx}
\usepackage{amssymb}
\usepackage{esint}
\usepackage[all]{xy}
\usepackage[unicode=true,pdfusetitle,
 bookmarks=true,bookmarksnumbered=false,bookmarksopen=false,
 breaklinks=false,pdfborder={0 0 1},backref=false,colorlinks=false]
 {hyperref}

\makeatletter
\makeatletter
\newtheorem*{rep@thm}{\rep@title}
\newcommand{\newreptheorem}[2]{%
\newenvironment{rep#1}[1][0,0]{%
\def\rep@title{#2 ##1}%
\begin{rep@thm}}%
{\end{rep@thm}}}
\makeatother
\theoremstyle{plain}
\newtheorem{thm}{Theorem}[section]
  \theoremstyle{plain}
  \newtheorem{prop}[thm]{Proposition}
  \newreptheorem{prop}{Proposition}
  \newreptheorem{thm}{Theorem}
  \theoremstyle{remark}
  \newtheorem*{acknowledgement*}{Acknowledgement}
  \theoremstyle{definition}
  \newtheorem{defn}[thm]{Definition}
  \theoremstyle{plain}
  \newtheorem*{cor*}{Corollary}
  \theoremstyle{plain}
  \newtheorem*{lem*}{Lemma}
  \theoremstyle{plain}
  \newtheorem*{prop*}{Proposition}
  \theoremstyle{plain}
  \newtheorem{cor}[thm]{Corollary}
  \theoremstyle{plain}
  \newtheorem{lem}[thm]{Lemma}
  \theoremstyle{remark}
  \newtheorem{rem}[thm]{Remark}
  \theoremstyle{plain}
  \newtheorem{conjecture}[thm]{Conjecture}

    \RequirePackage{rotating}                   
    \def\Hto{%
       \setbox0=\hbox{$\widehat{\mathit{HM}}$}
       \setbox1=\hbox{$\mathit{HM}$}
       \dimen0=1.1\ht0
       \advance\dimen0 by 1.17\ht1
       \smash{\mskip2mu\raise\dimen0\rlap{%
          \begin{turn}{180}
              {$\widehat{\phantom{\mathit{HM}}}$}
           \end{turn}} \mskip-2mu    
                \mathit{HM}
    }{\vphantom{\widehat{\mathit{HM}}}}{}}

\makeatother

\begin{document}

\title{Monopole Floer homology and Legendrian knots}

\author{Steven Sivek}
\maketitle
\begin{abstract}
We use monopole Floer homology for sutured manifolds to construct
invariants of Legendrian knots in a contact 3-manifold. These invariants
assign to a knot $K\subset Y$ elements of the monopole knot homology
$KHM(-Y,K)$, and they strongly resemble the knot Floer homology invariants
of Lisca, Ozsváth, Stipsicz, and Szabó. We prove several vanishing
results, investigate their behavior under contact surgeries, and use
this to construct many examples of non-loose knots in overtwisted
3-manifolds. We also show that these invariants are functorial with
respect to Lagrangian concordance.
\end{abstract}

\section{Introduction}

A knot $\mathcal{K}$ in a contact 3-manifold $(Y,\xi)$ is said to
be \emph{Legendrian} if the tangent vectors to $\mathcal{K}$ lie
in the contact planes $\xi$. In recent years, a variety of invariants
have been constructed to distinguish Legendrian knots which are topologically
identical. Notable examples include contact homology \cite{Eliashberg-CH},
especially the combinatorial version due to Chekanov \cite{Chekanov},
and invariant elements of knot Floer homology constructed using either
grid diagrams \cite{Legendrian-grid-diagrams} or open book decompositions
\cite{LOSS}; the last of these is due to Lisca, Ozsváth, Stipsicz,
and Szabó and is thus often called the {}``LOSS invariant.''

In order to construct a knot invariant from monopole Floer homology,
Kronheimer and Mrowka defined a monopole version of Juhász's sutured
Floer homology \cite{Juhasz-sutured} and declared the monopole knot
homology $KHM(Y,K)$ to be the sutured invariant of the complement
of $K$. It is natural to ask whether the LOSS invariant can be defined
in this setting, where the construction makes no use of Heegaard diagrams
or open books but instead proceeds by embedding the knot complement
in a closed $3$-manifold $\bar{Y}$ and computing $\Hto(\bar{Y})$
in certain $\mathrm{Spin}^{c}$-structures.

The goal of this paper is to present such a invariant. Namely, to
any Legendrian knot $\mathcal{K}\subset(Y,\xi)$ of topological type
$K$, we associate elements \[
\ell_{g}(\mathcal{K})\in KHM(-Y,K)\]
in monopole knot homology with local coefficients, which are invariant
up to automorphisms of $KHM$, for all integers $g\geq2$. (Conjecturally
these do not depend on $g$, so for convenience we shall omit it throughout
this introduction and the reader may fix any choice of $g$.) These
elements are obtained by choosing a particular contact structure $\bar{\xi}$
on the closed manifold $\bar{Y}$, so that $(Y-\mathcal{K},\xi|_{Y-\mathcal{K}})$
is a contact submanifold of $(\bar{Y},\bar{\xi})$, and letting $\ell_{g}(\mathcal{K})$
be the monopole contact invariant of $\bar{\xi}$.

The construction of $\ell(\mathcal{K})$ presents some advantages
and disadvantages over that of the LOSS invariant. It is hard to compute
in general, and it does not come with a natural bigrading the way
elements of knot Floer homology do. However, some vanishing and nonvanishing
results have very simple proofs, as do several theorems involving
contact surgery. For example:
\begin{repprop}[\ref{pro:overtwisted-torsion}]
If the complement of $\mathcal{K}$ is overtwisted, then $\ell(\mathcal{K})=0$.
\end{repprop}

\begin{repprop}[\ref{pro:stabilize2}]
Let $S_{+}(\mathcal{K})$ and $S_{-}(\mathcal{K})$ denote the positive
and negative stabilizations of a Legendrian knot $\mathcal{K}$. Then
$\ell(S_{+}S_{-}(\mathcal{K}))=0$ for all $\mathcal{K}$.
\end{repprop}
We expect something stronger to hold, namely that $\ell(S_{-}(\mathcal{K}))=\ell(\mathcal{K})$
and $\ell(S_{+}(\mathcal{K}))=0$, since the analogous statements
are true for the LOSS invariant. (See Conjecture \ref{con:stabilization}.)

\begin{repthm}[\ref{thm:+1-surgery}]
Let $\mathcal{K},\mathcal{S}\subset(Y,\xi)$ be disjoint, and let
$\mathcal{K}_{\mathcal{S}}\subset Y_{\mathcal{S}}$ denote the image
of $\mathcal{K}$ in the contact manifold $Y_{\mathcal{S}}$ obtained
by performing a contact $(+1)$-surgery on $\mathcal{S}$. Then there
is a map \[
KHM(-Y,K)\to KHM(-Y_{S},K_{S})\]
sending $\ell(\mathcal{K})$ to $\ell(\mathcal{K}_{\mathcal{S}})$.
\end{repthm}
These results are all known to be true for the LOSS invariant, as
are several consequences we will pursue in this paper. However, using
work of Mrowka and Rollin \cite{Mrowka-knots,Mrowka-cobordisms} on
the monopole contact invariant which is not known to be true in Heegaard
Floer homology, we can investigate one entirely new property of $\ell(\mathcal{K})$:
its behavior under Lagrangian concordance \cite{Chantraine}.
\begin{repthm}[\ref{thm:concordance-map}]
Let $\mathcal{K}_{0},\mathcal{K}_{1}\subset(Y,\xi)$ be Legendrian
knots, with $Y$ a homology $3$-sphere, and suppose that $\mathcal{K}_{0}$
is Lagrangian concordant to $\mathcal{K}_{1}$. Then there is a map
\[
KHM(-Y,K_{1})\to KHM(-Y,K_{0})\]
such that $\ell(\mathcal{K}_{1})\mapsto\ell(\mathcal{K}_{0})$.
\end{repthm}
The organization of this paper is as follows. In section \ref{sec:background}
we review the necessary background on sutured monopole homology and
the monopole contact invariant. We construct $\ell_{g}(\mathcal{K})$,
prove its invariance, and compute it for Legendrian unknots in section
\ref{sec:Legendrian-invariant}, and prove the vanishing theorems
mentioned above in section \ref{sec:vanishing}. In section \ref{sec:contact-surgery}
we investigate the effect of contact $(+1$)-surgery on $\ell(\mathcal{K})$,
and apply this to prove some nonvanishing results and to construct
many examples of non-loose knots in overtwisted contact manifolds.
Finally, in section \ref{sec:Lagrangian-concordance} we discuss the
behavior of $\ell(\mathcal{K})$ with respect to Lagrangian concordance.

Throughout this paper we will adopt the convention that letters in
the standard math font, such as $K$, refer to topological knots,
whereas the same letters in a script font, such as $\mathcal{K}$,
refer to Legendrian representatives of those knot types. We also remark
that Lekili \cite{Lekili-sfh} has shown that one can replace $\Hto$
with $HF^{+}$ in the Kronheimer-Mrowka construction of sutured monopole
homology in order to recover sutured Floer homology. Thus the reader
can apply the constructions in this paper to obtain a similar Legendrian
invariant in knot Floer homology, and everything in this paper will
still hold except the Lagrangian concordance results of section \ref{sec:Lagrangian-concordance}.
In this sense we conjecture that $\ell(\mathcal{K})$ is identical
to the LOSS invariant.
\begin{acknowledgement*}
An early version of this work formed part of my thesis at MIT under
the supervision of Tom Mrowka, who I thank for many useful discussions
and suggestions. I am also grateful to many others, in particular
Jon Bloom, John Etnyre, Peter Kronheimer, Yank\i{} Lekili, Lenny Ng,
Peter Ozsváth, Paul Seidel, Vera Vértesi, and Chris Wendl, for helpful
conversations on this work and related issues. This work was partially
supported by an NSF Graduate Research Fellowship.
\end{acknowledgement*}

\section{Sutured manifolds and contact invariants in monopole Floer homology\label{sec:background}}

\subsection{The definition of $SHM$}

For background on monopole Floer homology we refer to \cite{KM-book}.

Let $(M,\gamma)$ be a balanced sutured manifold. Kronheimer and Mrowka
\cite{KM-sutured} defined the monopole Floer homology of $(M,\gamma)$
as follows:
\begin{enumerate}
\item Choose an oriented, connected surface $T$ such that the components
of $\partial T$ are in one-to-one correspondence with the components
of $\gamma$. Form the product sutured manifold $(T\times I,\delta)$,
where $I=[-1,1]$, with annuli $A(\delta)=\partial T\times I$ and
$R_{\pm}(\delta)=T\times\{\pm1\}$.
\item Glue the annuli $A(\delta)$ to $A(\gamma)$ by some orientation-reversing
map $A(\delta)\to A(\gamma)$ sending $\partial R_{+}(\delta)$ to
$\partial R_{+}(\gamma)$. The resulting $3$-manifold should have
boundary $\bar{R}_{+}\cup\bar{R}_{-}$ for some connected, closed,
orientable surfaces $\bar{R}_{\pm}=R_{\pm}(\gamma)\cup R_{\pm}(\delta)$.
\item Form a closed manifold $\bar{Y}$ by gluing the boundary along some
diffeomorphism $h:\bar{R}_{+}\to\bar{R}_{-}$, and let $\bar{R}\subset\bar{Y}$
be the image of $\bar{R}_{\pm}$.
\end{enumerate}
We require that $\bar{R}$ has genus at least $2$, and that $T$
contains a simple closed curve $c$ such that $c\times\{\pm1\}$ is
a non-separating curve in $\bar{R}_{\pm}$.
\begin{defn}
The sutured monopole homology of $(M,\gamma)$ is defined as \[
SHM(M,\gamma)=\Hto_{\bullet}(\bar{Y}|\bar{R}),\]
where $\Hto_{\bullet}(\bar{Y}|\bar{R})$ is the direct sum of $\Hto_{\bullet}(\bar{Y},\mathfrak{s})$
over all $\mathrm{Spin}^{c}$ structures $\mathfrak{s}$ satisfying
$\langle c_{1}(\mathfrak{s}),\bar{R}\rangle=2g(\bar{R})-2$.
\end{defn}
Note that since $g(\bar{R})\geq2$, the class $c_{1}(\mathfrak{s})$
cannot be torsion if $\Hto_{\bullet}(\bar{Y},\mathfrak{s})$ contributes
to $\Hto_{\bullet}(\bar{Y}|\bar{R})$; but then $\overline{HM}(Y,\mathfrak{s})=0$,
so $\Hto_{\bullet}(\bar{Y},\mathfrak{s})$ and $\widehat{HM}_{\bullet}(\bar{Y},\mathfrak{s})$
are canonically isomorphic. In \cite{KM-sutured} the authors therefore
simply write $HM(\bar{Y}|\bar{R})$, but we will prefer to leave the
{}``to'' decoration in place as a reminder that we will be using
the contact invariant associated to $\Hto$.

We can also define $SHM$ using local coefficients. Let $\mathcal{R}$
be a ring with exponential map $\exp:\mathbb{R}\to\mathcal{R}^{\times}$
and write $t^{n}=\exp(n)$ for convenience. To any smooth 1-cycle
$\eta$ in $\bar{Y}$ we can associate a local system $\Gamma_{\eta}$
on the Seiberg-Witten configuration space $\mathcal{B}(\bar{Y},\mathfrak{s})$
whose fiber at any point is $\mathcal{R}$ and which assigns to any
path $z:[0,1]\to\mathcal{B}(\bar{Y},\mathfrak{s})$ the multiplication
map by $t^{r(z)}$, where \[
r(z)=\frac{i}{2\pi}\int_{[0,1]\times\eta}\mathrm{tr}(F_{A_{z}})\]
for $A_{z}$ the connection on $[0,1]\times\bar{Y}$ arising from
the path $z$.

Suppose that the diffeomorphism $h:\bar{R}_{+}\to\bar{R}_{-}$ restricts
to an orientation-preserving homeomorphism $c\times\{1\}\to c\times\{-1\}$,
resulting in a curve $\bar{c}\subset\bar{Y}$. If $\eta$ is taken
to be a curve dual to $\bar{c}$, in the sense that $\bar{c}\cdot\eta=1$,
then we can define $SHM(M,\gamma;\Gamma_{\eta})=\Hto_{\bullet}(\bar{Y}|\bar{R};\Gamma_{\eta})$.
As in the case without local coefficients, if $t-t^{-1}$ is invertible
in $\mathcal{R}$ then the authors simply write $HM(\bar{Y}|\bar{R};\Gamma_{\eta})$
without any ambiguity but we will continue to use $\Hto$.
\begin{prop}[\cite{KM-sutured}]
If $t-t^{-1}$ is invertible in $\mathcal{R}$, then $SHM(M,\gamma;\Gamma_{\eta})$
depends only on $(M,\gamma)$ and $\mathcal{R}$. In this case we
can allow $\bar{R}$ to have genus 1, but if $g(\bar{R})\geq2$ and
$\mathcal{R}$ has no $\mathbb{Z}$-torsion then we also have \[
SHM(M,\gamma;\Gamma_{\eta})\cong SHM(M,\gamma)\otimes\mathcal{R}.\]

\end{prop}

\subsection{$SHM$ with coefficients in $\mathbb{Z}/2\mathbb{Z}$}

Throughout \cite{KM-sutured} the authors work with coefficients (both
local and otherwise) in $\mathbb{Z}$; however, we assert that $SHM$
is still an invariant if $\mathbb{F}=\mathbb{Z}/2\mathbb{Z}$ is used
instead. When using systems of local coefficients $\Gamma_{\eta}$
over $\mathbb{F}$, we drop the condition that the ring $\mathcal{R}$
have no $\mathbb{Z}$-torsion and thus only require that $t-t^{-1}\in\mathcal{R}$
be invertible. This will allow us to pursue several applications involving
surgery exact triangles, which are known to work with local coefficients
over $\mathbb{F}$ (see \cite{KMOS-lens} or \cite[Chapter 42]{KM-book})
but which have not yet been proved with local coefficients over $\mathbb{Z}$.

The proofs of the invariance theorems in \cite{KM-sutured}, Theorem
4.4 and Proposition 4.6, rely on several facts, most notably the excision
theorems, Theorems 3.1 through 3.3, which still apply verbatim. We
need to verify that a handful of important proofs still work, and
in each case the only step requiring some additional care is the vanishing
of a $\mathrm{Tor}$ group coming from an application of the Künneth
theorem:
\begin{cor*}[{\cite[Corollary 3.4]{KM-sutured}}]
Let $\Sigma\subset Y$ be a closed, oriented surface, and let $\eta$
be a $1$-cycle supported in $\Sigma$. If every component of $\Sigma$
has genus at least $2$, then \[
\Hto(Y|\Sigma;\Gamma_{\eta})\cong\Hto(Y|\Sigma)\otimes\mathcal{R}.\]
\end{cor*}
\begin{proof}
The only detail requiring care in the original proof is the map (14),
denoted \[
HM_{\bullet}(Y_{1}|\Sigma_{1};\Gamma_{\eta_{0}})\otimes HM_{\bullet}(Y_{2}|\Sigma_{2})\to HM_{\bullet}(\tilde{Y}|\tilde{\Sigma};\Gamma_{\eta_{0}}),\]
which comes from an application of the Künneth theorem and is expected
to be an isomorphism. The cokernel of this map is \[
\mathrm{Tor}_{\mathbb{F}}(HM_{\bullet}(Y_{1}|\Sigma_{1};\Gamma_{\eta_{0}}),HM_{\bullet}(Y_{2}|\Sigma_{2})),\]
which is zero since $HM_{\bullet}(Y_{1}|\Sigma_{1};\Gamma_{\eta_{0}})=\mathcal{R}$
is a free $\mathbb{F}$-module, so the rest of the proof still applies.\end{proof}
\begin{lem*}[{\cite[Lemma 4.7]{KM-sutured}}]
Let $Y$ be fibered over $S^{1}$ with closed fiber $R$ of genus
at least $2$. Then $\Hto(Y|R)\cong\mathbb{F}$.\end{lem*}
\begin{proof}
As before, if $Y_{h}$ is the mapping torus of $h:R\to R$ then $\Hto(Y_{h}|R)\cong\Hto(Y_{h^{-1}}|R)$
and so the excision theorem applied to $Y_{h}\sqcup Y_{h^{-1}}$ gives
an injective map $\Hto(Y_{h}|R)\otimes_{\mathbb{F}}\Hto(Y_{h}|R)\to\mathbb{F}$
with cokernel \[
\mathrm{Tor}_{\mathbb{F}}(\Hto(Y_{h}|R),\Hto(Y_{h}|R)).\]
Since $\Hto(Y_{h}|R)$ is a free $\mathbb{F}$-module, this $\mathrm{Tor}$
term vanishes and the map is an isomorphism, hence $\Hto(Y_{h}|R)\cong\mathbb{F}$.\end{proof}
\begin{cor*}[{\cite[Corollary 4.8]{KM-sutured}}]
The sutured homology group $SHM(M,\gamma)$ does not depend on the
choice of gluing homeomorphism $h$.\end{cor*}
\begin{proof}
This is another application of the excision theorem, Theorem 3.1,
to a disconnected manifold $Y=Y_{1}\sqcup Y_{2}$ with $Y_{2}$ a
mapping torus, hence $\Hto(Y_{2}|\Sigma_{2})\cong\mathbb{F}$ and
again the proof is the same once we observe that \[
\mathrm{Tor}_{\mathbb{F}}(\Hto(Y_{1}|\Sigma_{1}),\Hto(Y_{2}|\Sigma_{2}))\cong0.\]
\end{proof}
\begin{prop*}[{\cite[Proposition 4.10]{KM-sutured}}]
If $t-t^{-1}$ is invertible in $\mathcal{R}$, then $SHM(M,\gamma;\Gamma_{\eta})$
is independent of the genus $g$.\end{prop*}
\begin{proof}
Here we wish to show that \[
\Hto(Y_{1}|\bar{R}_{1};\Gamma_{\eta_{1}})\cong\Hto((Y_{1}\sqcup Y_{2})|(\bar{R}_{1}\sqcup\bar{R}_{2});\Gamma_{\eta})\]
where $\eta=\eta_{1}+\eta_{2}$ for some cycles $\eta_{i}\subset\bar{R}_{i}\subset Y_{i}$,
and we know that $\Hto(Y_{2}|\bar{R}_{2};\Gamma_{\eta_{2}})=\mathcal{R}$.
The Künneth theorem thus gives a map \[
\Hto(Y_{1}|\bar{R}_{1};\Gamma_{\eta_{1}})\otimes_{\mathcal{R}}\mathcal{R}\to\Hto((Y_{1}\sqcup Y_{2})|(\bar{R}_{1}\sqcup\bar{R}_{2});\Gamma_{\eta})\]
with cokernel \[
\mathrm{Tor}_{\mathcal{R}}(\Hto(Y_{1}|\bar{R}_{1};\Gamma_{\eta_{1}}),\mathcal{R}),\]
and since $\mathcal{R}$ is free as an $\mathcal{R}$-module, the
$\mathrm{Tor}$ term vanishes and this is indeed an isomorphism.
\end{proof}
We conclude that both the standard and local versions of sutured monopole
homology are invariants if we work over $\mathbb{F}$ rather than
$\mathbb{Z}$.

\subsection{Monopole knot homology}

Given a knot $K$ in a closed, oriented 3-manifold $Y$, we can form
a sutured manifold $Y(K)=(M,\gamma)$ as in \cite{Juhasz-sutured}
by taking $M$ to be the knot complement $Y\backslash N(K)$ and $\gamma\subset\partial M$
a pair of oppositely oriented meridians. Then monopole knot homology
is defined by \[
KHM(Y,K)=SHM(M,\gamma),\]
and if we work with local coefficients we get $KHM(Y,K)\otimes\mathcal{R}\cong SHM(M,\gamma;\Gamma_{\eta})$. 

From now on we will fix $\mathcal{R}$ to be the Novikov ring \[
\left\{ \left.\sum_{\alpha}c_{\alpha}t^{\alpha}\right\vert \alpha\in\mathbb{R},\ c_{\alpha}\in\mathbb{F},\ \#\{\beta<n\mid c_{\beta}\not=0\}<\infty\ \mathrm{for\ all}\ n\right\} ,\]
with $\exp(\alpha)=t^{\alpha}$ and $(t-t^{-1})^{-1}=-t-t^{3}-t^{5}-\dots$.
Although we may drop the local system $\Gamma_{\eta}$ from our notation,
we are always working with local coefficients over $\mathcal{R}$.

\subsection{Contact structures in monopole Floer homology}

Let $(Y,\xi)$ be a closed contact 3-manifold. Kronheimer and Mrowka
\cite{KM-contact} associate a contact invariant \[
\psi(\xi):\Lambda(\xi)\to\Hto_{\bullet}(-Y,\mathfrak{s}_{\xi},c_{\mathrm{bal}},\Gamma_{\eta})\]
where $c_{\mathrm{bal}}=2\pi c_{1}(\mathfrak{s}_{\xi})$ is a balanced
perturbation of the Seiberg-Witten equations and $\Lambda(\xi)$ is
the set of orientations of an appropriate moduli space. In general
we will ignore the orientations $\Lambda(\xi)$, since we are working
in characteristic 2, and so we will write $\psi(\xi)\in\Hto_{\bullet}(-Y,\mathfrak{s}_{\xi},c_{\mathrm{bal}},\Gamma_{\eta})$.

Mrowka and Rollin \cite{Mrowka-knots,Mrowka-cobordisms} investigated
the behavior of the contact invariant under symplectic cobordisms.
\begin{defn}
A symplectic cobordism $(W,\omega)$ from $(Y_{-},\xi_{-})$ to $(Y_{+},\xi_{+})$
is said to be \emph{left-exact} if $\omega$ is exact near $Y_{-}$,
or equivalently if it is given in a collar neighborhood of $Y_{-}$
by a symplectization $\frac{1}{2}d(t^{2}\eta_{-})$ where $\xi_{-}=\ker\eta_{-}$.
It is \emph{right-exact} if the same holds near $Y_{+}$, and \emph{boundary-exact}
if it is both left- and right-exact.\end{defn}
\begin{thm}[{\cite[Theorem 3.5.4]{Mrowka-cobordisms}}]
\label{thm:symplectic-cobordism}Let $W$ be a boundary-exact cobordism
$(W,\omega)$ as above such that the map \[
H^{1}(W;\mathbb{Z})\to H^{1}(Y_{+};\mathbb{Z})\]
is surjective, and let $W^{\dagger}$ denote $W$ viewed as a cobordism
from $-Y_{+}$ to $-Y_{-}$. Then $\psi(\xi_{-})=\Hto(W^{\dagger},\mathfrak{s}_{\omega})(\psi(\xi_{+}))$.\end{thm}
\begin{cor}
\label{cor:surgery-cobordism}If $W$ is a symplectic 2-handle cobordism
corresponding to Legendrian surgery, then \[
\psi(\xi_{-})=\Hto(W^{\dagger})\psi(\xi_{+})\]
and in particular we can replace the map $\Hto(W^{\dagger},\mathfrak{s}_{\omega})$
of Theorem \ref{thm:symplectic-cobordism} by the total map $\Hto(W^{\dagger})$.
\end{cor}

\section{The Legendrian knot invariant\label{sec:Legendrian-invariant}}

Let $\mathcal{K}\subset(Y,\xi)$ be a Legendrian knot of topological
knot type $K$. Our goal is to construct an appropriate contact structure
$\bar{\xi}$ on a closure $(\bar{Y},\bar{R})$ of the sutured knot
complement $Y(K)$ so that the contact invariant $\psi(\bar{\xi})$
does not depend on any of the choices we must make. This will give
us an invariant $\ell(\mathcal{K})$ of the Legendrian knot $\mathcal{K}$
which is an element of $KHM(-Y,K)$ up to automorphism.

Take a standard neighborhood $N(\mathcal{K})$ whose boundary is a
convex torus. If we assign coordinates to $\partial N(\mathcal{K})\cong\mathbb{R}^{2}/\mathbb{Z}^{2}$
so that $(\pm1,0)$ is a meridian and $(0,\pm1)$ a longitude, then
its dividing set $\Gamma$ consists of two parallel curves of slope
$\frac{1}{tb(\mathcal{K})}$, where $tb(\mathcal{K})$ is the Thurston-Bennequin
invariant of $\mathcal{K}$. (See for example \cite[Section 2.4]{Etnyre-nonexistence}.)
In particular, if we view the sutured knot complement $Y(\mathcal{K})$
as the contact manifold $(Y\backslash N(\mathcal{K}),\xi|_{Y\backslash N(\mathcal{K})})$
with convex boundary, then each of the meridional sutures will intersect
each dividing curve transversely in a single point as in Figure \ref{fig:sutures-and-dividing-curves}.
Here, and in all other figures, we will color regions white and gray
to represent the positive and negative regions, respectively, of a
convex surface.

\begin{figure}
\begin{centering}
\includegraphics{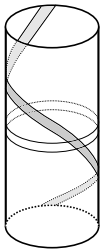}
\par\end{centering}

\caption{The convex torus $\partial(Y\backslash N(\mathcal{K}))$, cut along
a meridian. The horizontal circles are sutures, while the pair of
parallel arcs (or circles, once the top and bottom are identified)
are dividing curves and have slope $\frac{1}{tb(\mathcal{K})}$.\label{fig:sutures-and-dividing-curves}}
\end{figure}

\subsection{\label{sub:construct-closure}Closure of the sutured knot complement}

Our construction follows the definition of the sutured invariant as
in Section \ref{sec:background}. We must first pick an auxiliary
surface $T$ whose boundary components are in one-to-one correspondence
with the sutures of $Y(\mathcal{K})$ and glue the annuli $\partial T\times I\subset T\times I$
to neighborhoods $A(\gamma)$ of the sutures. In order to form a contact
structure on this glued manifold, we must assign a contact structure
to $T\times I$ whose restriction to a neighborhood of $\partial T\times I$
agrees with $\xi$ in a neighborhood of $A(\gamma)$. By Giroux's
flexibility theorem \cite{Giroux-convexite} it suffices to ensure
that the dividing curves match, sending the positive region of $A(\gamma)$
to the negative region of $\partial T\times I$ and vice versa.

Let $T_{0}$ be a closed surface of genus at least 2, and pick a pair
of dual curves $\alpha,\beta\subset T_{0}$ such that $\alpha\cdot\beta=1$.
Give $T_{0}\times I$ the $I$-invariant contact structure $\Xi_{\alpha}$
whose dividing curves consist of two parallel disjoint copies of $\alpha$
on each surface $T_{0}\times\{\pm1\}$ and for which the negative
region of $T_{0}\times\{+1\}$ is an annulus. We define $(T\times I,\Xi)$
as the contact manifold obtained by cutting $T_{0}\times I$ along
a convex perturbation of the annulus $\beta\times I$, as in Figure
\ref{fig:contact-surface}; it may also be viewed as a product sutured
manifold with sutures $\delta=\partial T\times\{0\}$.

\begin{figure}
\begin{centering}
\includegraphics[scale=0.95]{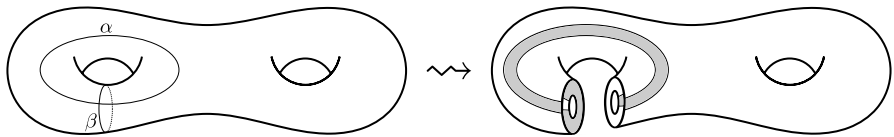}
\par\end{centering}

\caption{The construction of the auxiliary surface $(T\times I,\Xi)$.\label{fig:contact-surface}}

\end{figure}

We now form a contact manifold $(Y',\xi')=(Y\backslash N(\mathcal{K}))\cup(T\times I)$
by gluing along some orientation-reversing diffeomorphism $A(\delta)\to A(\gamma)$
as described above. This manifold has edges, corresponding to the
corners $\partial T\times\partial I$, which we smooth using edge-rounding
\cite{Honda-tight1}, under which dividing curves turn to the left
(as viewed from outside $Y'$) as they approach an edge. See Figure
\ref{fig:edge-rounding}.

\begin{figure}
\begin{centering}
\includegraphics{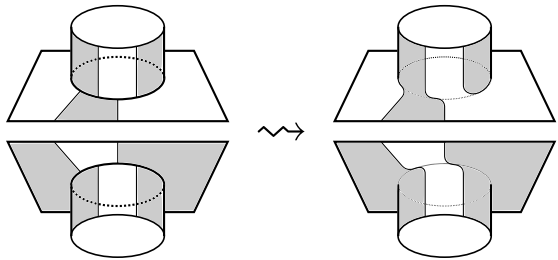}
\par\end{centering}

\caption{Gluing $T\times I$ to $Y\backslash N(\mathcal{K})$ and rounding
edges in a cylindrical neighborhood of one of the sutures on $\partial N(\mathcal{K})$,
as viewed from $T\times\{1\}$ on top and $T\times\{-1\}$ on the
bottom.\label{fig:edge-rounding}}

\end{figure}

\begin{lem}
The contact manifold $Y'=(Y\backslash N(\mathcal{K}))\cup(T\times I)$
depends only on $\mathcal{K}$, $(Y,\xi)$, and the genus of $T_{0}$.\end{lem}
\begin{proof}
The construction of $T\times I$ depends only on $g(T_{0})$ and on
the curves $\alpha,\beta\subset T_{0}$. Given any other pair of curves
$\alpha'$ and $\beta'$ which intersect once, there is a diffeomorphism
$\varphi:T_{0}\to T_{0}$ with $\varphi(\alpha)=\alpha'$ and $\varphi(\beta)=\varphi(\beta')$,
and this extends to a contactomorphism $\varphi\times Id:(T_{0}\backslash N(\beta))\times I\to(T_{0}\backslash N(\beta'))\times I$.
\end{proof}
Finally, we close up $Y'$ to get a contact manifold $(\bar{Y},\bar{\xi})$
with distinguished convex surface $\bar{R}$ of genus $g\geq2$. The
boundary of $Y'$ consists of two convex surfaces $\bar{R}_{+}$ and
$\bar{R}_{-}$ determined by $T\times\{\pm1\}\subset\bar{R}_{\pm}$.
These are split by pairs of parallel dividing curves $\Gamma_{\pm}\subset\bar{R}_{\pm}$
into positive and negative regions $(\bar{R}_{+})_{\pm}\subset\bar{R}_{+}$
and $(\bar{R}_{-})_{\pm}\subset\bar{R}_{-}$, and each of $(\bar{R}_{+})_{-}$
and $(\bar{R}_{-})_{+}$ is an annulus. Fix any diffeomorphism $h:\bar{R}_{+}\to\bar{R}_{-}$
which sends $(\bar{R}_{+})_{\pm}$ to $(\bar{R}_{-})_{\mp}$, and
hence also $\Gamma_{+}$ to $\Gamma_{-}$, and such that $h(x\times\{1\})$
and $x\times\{-1\}$ lie in the same component of $\Gamma_{-}$ for
any $x\times\{1\}$ in $\Gamma_{+}\cap(\mathrm{int}(T)\times\{1\})$.
In other words, a dividing curve $c\subset\Gamma_{+}$ corresponds
to one of the two copies of $\alpha\subset T_{0}$, and we require
$h(c)$ to be the dividing curve of $\Gamma_{-}$ corresponding to
the same copy of $\alpha$.

We glue $\bar{R}_{+}$ to $\bar{R}_{-}$ via $h$. The resulting contact
manifold is the desired $(\bar{Y},\bar{\xi})$.
\begin{defn}
The Legendrian invariant of $\mathcal{K}$ is defined as $\ell(\mathcal{K})=\psi(\bar{Y},\bar{\xi})\in\Hto(-\bar{Y},\mathfrak{s}_{\bar{\xi}};\Gamma_{\eta})$.
\end{defn}
We can compute that\[
\langle c_{1}(\bar{\xi}),\bar{R}\rangle_{\bar{Y}}=\chi((\bar{R}_{+})_{+})-\chi((\bar{R}_{+})_{-})=2-2g(\bar{R})\]
and so $\langle c_{1}(\mathfrak{s}_{\bar{\xi}}),\bar{R}\rangle_{-\bar{Y}}=2g(\bar{R})-2$.
This means that $\ell(\mathcal{K})$ is in fact an element of $\Hto(-\bar{Y}|\bar{R};\Gamma_{\eta})=SHM(-Y(K);\Gamma_{\eta})$,
which is by definition the knot homology with local coefficients.
Therefore \[
\ell(\mathcal{K})\in KHM(-Y,K)\otimes\mathcal{R}.\]

\begin{rem}
The desire to arrange that $\langle c_{1}(\bar{\xi}),\bar{R}\rangle=2-2g$
motivated our choice of contact structure on $T\times I$. In particular,
$\chi((\bar{R}_{+})_{+})$ and $\chi((\bar{R}_{+})_{-})$ sum to $\chi(\bar{R}_{+})=2-2g$,
and if we fix their difference as above then we must have $\chi((\bar{R}_{+})_{-})=0$.
But now $(\bar{R}_{+})_{-}$ does not have any sphere or torus components,
and if it had disk components then $\bar{R}_{+}$ would not have a
tight neighborhood \cite{Giroux-convexite}, so $(\bar{R}_{+})_{-}$
is forced to be a union of annuli.
\end{rem}
In addition to the Legendrian knot $\mathcal{K}\subset(Y,\xi)$, the
construction of $\ell(\mathcal{K})$ from a closure $(\bar{Y},\bar{\xi})$
with distinguished convex surface $\bar{R}$ potentially depends on
both the genus $g=g(\bar{R})$ and the choice of diffeomorphism $\bar{R}_{+}\to\bar{R}_{-}$.
Our goal in the next section is to prove that in fact it is independent
of the diffeomorphism.

\subsection{Invariance under diffeomorphism}

In this section we establish that $\ell(\mathcal{K})$ is independent
of the choice of diffeomorphism $\bar{R}_{+}\to\bar{R}_{-}$.
\begin{prop}
\label{pro:diffeo-invariance}Let $(\bar{Y}',\bar{\xi}')$ be the
contact manifold obtained from $\bar{Y}$ by cutting along the convex
surface $\bar{R}$ and regluing along some orientation-preserving
diffeomorphism $h$ such that $h(\gamma)=\gamma$ for each dividing
curve $\gamma$ of $\bar{R}$. Then there is an isomorphism $\Hto(-\bar{Y}'|\bar{R};\Gamma_{\eta})\to\Hto(-\bar{Y}|\bar{R};\Gamma_{\eta})$
which sends $\psi(\bar{Y}',\bar{\xi}')$ to $\psi(\bar{Y},\bar{\xi})$.\end{prop}
\begin{lem}
\label{lem:dehn-twist}Proposition \ref{pro:diffeo-invariance} holds
when $h$ is a Dehn twist along some nonseparating curve $c$ which
does not intersect the dividing curves $\Gamma$ of $\bar{R}$.\end{lem}
\begin{proof}
We observe that $c$ is nonisolating, i.e. that every component of
$\bar{R}\backslash(\Gamma\cup c)$ has a boundary component which
intersects $\Gamma$, and thus by the Legendrian Realization Principle
\cite{Kanda-LeRP,Honda-tight1} we can take $c$ to be Legendrian.
Indeed, the complement $\bar{R}\backslash\Gamma$ has two connected
components; if $c$ is nonseparating within its component then it
is clearly nonisolating. Otherwise, $c$ divides its component of
$\bar{R}\backslash\Gamma$ into two components, say $A$ and $B$.
Since $c$ is nonseparating in $\bar{R}$ there is a path in $\bar{R}\backslash c$
which connects $A$ to $B$, and this path must pass through the other
component of $\bar{R}\backslash\Gamma$. In particular, the path crosses
$\Gamma$, so both $\partial A$ and $\partial B$ intersect $\Gamma$
and thus $c$ is nonisolating.

Suppose now that $h$ is a positive Dehn twist along $c$, and that
$c$ has been realized as a Legendrian curve. Then $h$ can be realized
by $(-1)$-surgery on $c$ with respect to the framing induced by
$\bar{R}$, and since $tw(c,\bar{R})=-\frac{1}{2}|c\cap\Gamma|=0$
this is a Legendrian surgery. If $W$ is the corresponding symplectic
cobordism, and $W^{\dagger}$ is the oppositely oriented cobordism
from $-\bar{Y}'$ to $-\bar{Y}$, then \[
\Hto(W^{\dagger})(\psi(\bar{Y}',\bar{\xi}'))=\psi(\bar{Y},\bar{\xi})\]
by Corollary \ref{cor:surgery-cobordism}. The fact that $\Hto(W^{\dagger})$
gives an isomorphism $\Hto(-\bar{Y}'|\bar{R};\Gamma_{\eta})\to\Hto(-\bar{Y}|\bar{R};\Gamma_{\eta})$
is an easy consequence of the surgery exact triangle for $\Hto$ and
the fact that $\bar{R}$ becomes compressible in the manifold $-\bar{Y}_{0}$
obtained by $0$-surgery along $c$, hence $\Hto(-\bar{Y}_{0}|\bar{R})=0$
by the adjunction inequality \cite{KM-book}.

If instead $h$ is a negative Dehn twist, we note that $\bar{Y}$
can be obtained from $\bar{Y}$' by a positive Dehn twist along $c$,
so we construct a cobordism $W$ from $\bar{Y}'$ to $\bar{Y}$ as
above and then $\Hto(W^{\dagger})^{-1}$ is the desired isomorphism. 
\end{proof}

\begin{proof}[Proof of Proposition \ref{pro:diffeo-invariance}]
In general, we can arrange by an isotopy that the diffeomorphism
$h$ is actually the identity on each dividing curve. Then $h$ restricts
to a boundary-fixing diffeomorphism on the closure of each component
of $\bar{R}\backslash\Gamma$. One component is an annulus $A$, so
up to isotopy $h|_{A}$ is a composition of Dehn twists about the
core of $A$. The other component is a surface $\Sigma$ of genus
$g(\bar{R})-1\geq1$ with two boundary components, and so $h|_{\Sigma}$
can also be expressed as a product of Dehn twists about nonseparating
curves which do not intersect $\Gamma=\partial\Sigma$. Since $h=h|_{A}\circ h|_{\Sigma}$,
repeated application of Lemma \ref{lem:dehn-twist} completes the
proof of Proposition \ref{pro:diffeo-invariance}.
\end{proof}
We have now shown that the construction of $\ell(\mathcal{K})\in KHM(-Y,K)\otimes\mathcal{R}$
is independent of all choices except possibly the genus $g=g(\bar{R})$.
Thus we have constructed a sequence of Legendrian knot invariants
$\ell_{g}(\mathcal{K})$ for $g\geq2$.
\begin{conjecture}
The elements $\ell_{g}(\mathcal{K})$, $g\geq2$, are all equal as
elements of $KHM(-Y,K)\otimes\mathcal{R}$ up to automorphism.
\end{conjecture}
Since we will show in section \ref{sub:Legendrian-unknot} that $\ell_{g}(\mathcal{U})=1\in\mathcal{R}$
where $\mathcal{U}\subset(S^{3},\xi_{\mathrm{std}})$ is the Legendrian
unknot with $tb=-1$, this conjecture would follow from a connected
sum formula of the form \[
\ell_{g}(\mathcal{K})\otimes\ell_{g'}(\mathcal{K}')=\ell_{g+g'-1}(\mathcal{K}\#\mathcal{K}')\]
which we expect to be true by comparison with the LOSS invariant \cite[Theorem 7.1]{LOSS}.

From now on we will drop the $g$ subscript where convenient and simply
write $\ell(\mathcal{K})$ to mean $\ell_{g}(\mathcal{K})$ for some
fixed $g\geq2$.

\subsection{The Legendrian unknot\label{sub:Legendrian-unknot}}

The Legendrian representatives of the topological unknot $U\subset S^{3}$
were classified by Eliashberg and Fraser \cite{Eliashberg-Fraser}:
they are completely determined by their classical invariants $tb$
and $r$, and there is a representative $\mathcal{U}$ with $(tb,r)=(-1,0)$
so that all others are stabilizations of $\mathcal{U}$. In this subsection
we will prove that the Legendrian invariant of $\mathcal{U}$ is a
unit of $KHM(U)\otimes\mathcal{R}\cong\mathcal{R}$.

Our strategy is to explicitly determine the contact structure on a
particular closure $\bar{Y}$ of $S^{3}(\mathcal{U})$.
\begin{lem}
Let $\xi$ be the $I$-invariant contact structure on $(S^{1}\times I)\times I$
whose dividing curves on the annulus $S^{1}\times I$ are a pair of
parallel arcs $\{t_{1}\}\times I$ and $\{t_{2}\}\times I$. Then
after edge-rounding, $\xi$ is contactomorphic to the complement of
$\mathcal{U}$.\end{lem}
\begin{proof}
By the classification of tight contact structures on solid tori \cite[Theorem 2.3]{Honda-tight1},
there is a unique tight contact structure $\Xi$ on $S^{1}\times D^{2}$
for which the dividing curves on the boundary have slope $-1$; since
$tb(\mathcal{U})=-1$, the complement of $\mathcal{U}$ must be $(S^{1}\times D^{2},\Xi)$.
But if we round the edges on $((S^{1}\times I)\times I,\xi)$, we
get a tight contact structure on $S^{1}\times D^{2}$ for which the
dividing curves on the boundary $S^{1}\times S^{1}$ have slope $-1$,
and so this contact structure must be $\Xi$ as well.\end{proof}
\begin{prop}
\label{pro:unknot-invariant}The invariant $\ell(\mathcal{U})$ is
a unit in $KHM(U)\otimes\mathcal{R}\cong\mathcal{R}$.\end{prop}
\begin{proof}
We glue a surface $T\times I$ to $(S^{1}\times I)\times I$ as in
Section \ref{sub:construct-closure}, identifying the annuli $\partial T\times I$
with $(S^{1}\times I)\times\partial I$, to get an $I$-invariant
contact manifold $Y'=\Sigma_{g}\times I$ which is universally tight
by Giroux's criterion \cite{Giroux-convexite} and has convex boundary.
Gluing $\Sigma_{g}\times\{1\}$ to $\Sigma_{g}\times\{-1\}$ via the
identity map, we get the closure $\bar{Y}=\Sigma_{g}\times S^{1}$
with $S^{1}$-invariant, universally tight contact structure $\bar{\xi}$
and distinguished surface $\bar{R}=\Sigma_{g}\times\{*\}$. Since
no component of $\Gamma\subset\Sigma_{g}$ is separating, Theorem
5 of \cite{Niederkruger-Wendl} asserts that $\bar{\xi}$ is weakly
fillable.

The claim that $KHM(U)\otimes\mathcal{R}=\Hto(\bar{Y}|\bar{R};\Gamma_{\eta})\cong\mathcal{R}$
now follows from Corollary 2.3 of \cite{KM-sutured}. Furthermore,
since $\bar{\xi}$ is weakly fillable we know that $\psi(\bar{\xi})$
is a unit of $\Hto(\bar{Y};\Gamma_{\eta})$ \cite{KM-contact,Mrowka-cobordisms},
and since $\psi(\bar{\xi})\in\Hto(\bar{Y}|\bar{R};\Gamma_{\eta})\cong\mathcal{R}$
the proposition follows.\end{proof}
\begin{rem}
Wendl \cite[Corollary 2]{Wendl-filling} has shown that $(\bar{Y},\bar{\xi})$
has vanishing untwisted ECH contact invariant. By work of Taubes \cite{Taubes-ECH5}
it follows that the untwisted contact invariant $\psi(\bar{\xi})\in\Hto(\bar{Y}|\bar{R})$
is also zero, so we must work with local coefficients for $\ell(\mathcal{U})$
to be nonzero.
\end{rem}
We can also compute $\ell(\mathcal{U}_{Y})$ if $\mathcal{U}_{Y}$
is a Legendrian unknot in a Darboux ball of some contact manifold
$(Y,\xi)$. Observe that both $S^{3}(U)$ and $S^{3}(1)$ have $(\Sigma_{g}\times S^{1},\Sigma_{g}\times\{*\})$
as a closure, where $M(1)$ denotes the complement of a ball in $M$
with a single suture, and since $Y(U_{Y})\cong Y\#S^{3}(U)$ we conclude
that $KHM(Y,U_{Y})\cong SHM(Y(1))$.

Let $\widetilde{HM}(Y)=SHM(Y(1))\cong KHM(Y,U_{Y})$. Then clearly
$\widetilde{HM}$ is analogous to the hat version of Heegaard Floer
homology, since $\widehat{HF}(Y)\cong SFH(Y(1))$ virtually by definition
\cite{Juhasz-sutured}. In fact, it is equivalent to define $\widetilde{HM}(Y)$
as the homology of the mapping cone of $\check{C}(Y)\stackrel{U_{\dagger}}{\to}\check{C}(Y)$
(\cite{Bloom-Kunneth}), just as $\widehat{HF}(Y)$ comes from the
Heegaard Floer complex $CF^{+}(Y)$.

We will define a contact invariant $\tilde{\psi}_{g}(\xi)\in\widetilde{HM}(-Y)\otimes\mathcal{R}$
up to automorphism as $\ell_{g}(\mathcal{U}_{Y})$. (Having noted
that $\tilde{\psi}_{g}$ potentially depends on $g$ just as $\ell_{g}$
does, we will similarly drop the subscript and write $\tilde{\psi}(\xi)$
from now on.) This seems to be a reasonable choice by analogy with
the LOSS invariant $\hat{\mathcal{L}}(\mathcal{U}_{Y})\in\widehat{HFK}(-Y,U_{Y})$,
which is identified with the Heegaard Floer contact invariant $\hat{c}(\xi)\in\widehat{HF}(-Y)$
as argued in the proof of \cite[Corollary 7.3]{LOSS}.
\begin{prop}
\label{pro:general-unknot}There is a map \[
\widetilde{HM}(-Y)\otimes_{\mathbb{F}}\mathcal{R}\to\Hto(-Y)\otimes_{\mathbb{F}}\mathcal{R}\]
which sends $\tilde{\psi}(\xi)=\ell(\mathcal{U}_{Y})$ to $\psi(\xi)\otimes1$.\end{prop}
\begin{proof}
Recall that $\mathcal{U}\subset(S^{3},\xi_{\mathrm{std}})$ has closure
$(\bar{Y},\bar{R})=(\Sigma_{g}\times S^{1},\Sigma_{g}\times\{*\})$
with $S^{1}$-invariant contact structure $\bar{\xi}$ and its homology
is twisted by a 1-cycle $\eta\subset\bar{Y}$. Thus the Legendrian
unknot $\mathcal{U}_{Y}\subset Y$ has closure $(Y\#\bar{Y},\bar{R})$. 

Build a symplectic cobordism $(W,\omega)$ from $(Y,\xi)\sqcup(\bar{Y},\bar{\xi})$
to $(Y\#\bar{Y},\xi\#\bar{\xi})$ by attaching a symplectic 1-handle
to the symplectization $(Y\sqcup\bar{Y})\times I$. The induced map
\[
\Hto(W_{\dagger},\mathfrak{s}_{\omega}):\Hto(-(Y\#\bar{Y}),\mathfrak{s}_{\xi}\#\mathfrak{s}_{\bar{\xi}};\Gamma_{\eta})\to\Hto(-Y\sqcup-\bar{Y},\mathfrak{s}_{\xi}\sqcup\mathfrak{s}_{\bar{\xi}};\Gamma_{\eta})\]
sends $\ell(\mathcal{U}_{Y})=\psi(\xi\#\bar{\xi})\in\Hto(-(Y\#\bar{Y});\Gamma_{\eta})$
to $\psi(\xi\sqcup\bar{\xi})$ by Theorem \ref{thm:symplectic-cobordism}.
But the map \[
\Hto(-Y,\mathfrak{s}_{\xi})\otimes_{\mathbb{F}}\Hto(-\bar{Y},\mathfrak{s}_{\bar{\xi}};\Gamma_{\eta})\to\Hto(-Y\sqcup-\bar{Y},\mathfrak{s}_{\xi}\sqcup\mathfrak{s}_{\bar{\xi}};\Gamma_{\eta})\]
coming from the Künneth theorem is an isomorphism since $\Hto(-\bar{Y},\mathfrak{s}_{\bar{\xi}};\Gamma_{\eta})\cong\mathcal{R}$
is free, so in fact $\Hto(W_{\dagger},\mathfrak{s}_{\omega})$ can
be expressed as a map \[
\Hto(-(Y\#\bar{Y}),\mathfrak{s}_{\xi}\#\mathfrak{s}_{\bar{\xi}};\Gamma_{\eta})\to\Hto(-Y,\mathfrak{s}_{\xi})\otimes_{\mathbb{F}}\mathcal{R}\]
sending $\psi(\xi\#\bar{\xi})$ to $\psi(\xi)\otimes\psi(\bar{\xi})=\psi(\xi)\otimes1$.
The source and target of this map are summands of $\Hto(-(Y\#\bar{Y})|\bar{R};\Gamma_{\eta})=\widetilde{HM}(-Y)\otimes\mathcal{R}$
and $\Hto(-Y)\otimes\mathcal{R}$, respectively, and $\psi(\xi\#\bar{\xi})$
is $\ell(\mathcal{U}_{Y})$, so we are done.\end{proof}
\begin{cor}
If $\psi(\xi)\in\Hto(-Y)$ is nonzero, then so is $\tilde{\psi}(\xi)\in\widetilde{HM}(-Y)\otimes\mathcal{R}$.
\end{cor}
For example, if $\xi$ is strongly symplectically fillable then $\psi(\xi)$
is nonzero and primitive \cite{KM-contact,Mrowka-knots}, so Proposition
\ref{pro:general-unknot} implies that $\tilde{\psi}(\xi)$ is a primitive
element of $\widetilde{HM}(-Y)\otimes\mathcal{R}$.

\section{Vanishing results\label{sec:vanishing}}

\subsection{Loose knots}

Recall that a Legendrian knot $\mathcal{K}\subset(Y,\xi)$ is said
to be \emph{loose} if the complement of $\mathcal{K}$ is overtwisted.
\begin{prop}
\label{pro:overtwisted-torsion}If $\mathcal{K}\subset Y$ is loose,
then $\ell(\mathcal{K})=0$.\end{prop}
\begin{proof}
By assumption $Y\backslash\mathcal{K}$ has an overtwisted disk, so
any closure $(\bar{Y},\bar{\xi})$ does as well. Then $\psi(\bar{Y},\bar{\xi})$
vanishes (see \cite[Corollary B]{Mrowka-knots}), hence $\ell(\mathcal{K})$
does as well.
\end{proof}

\subsection{Stabilization}

Let $S_{+}(\mathcal{K})$ and $S_{-}(\mathcal{K})$ denote the positive
and negative stabilizations of a Legendrian knot $\mathcal{K}$, which
may also be thought of as the connected sums $\mathcal{K}\#\mathcal{U}_{\pm}$
where $\mathcal{U}_{\pm}\subset S^{3}$ is the topologically trivial
knot with $tb=-2$ and $r=\pm1$. We expect the following conjecture
to be true:
\begin{conjecture}
\label{con:stabilization}For any Legendrian knot $\mathcal{K}\subset Y$
we have $\ell(S_{-}(\mathcal{K}))=\ell(\mathcal{K})$ and $\ell(S_{+}(\mathcal{K}))=0$.
\end{conjecture}
A theorem of Epstein, Fuchs, and Meyer \cite{Transverse-approximations}
characterizes transverse knots as pushoffs of Legendrian knots up
to negative stabilization, and so $\ell$ would then give a transverse
knot invariant as well.

We can prove a slightly weaker result than the desired $\ell(S_{+}(\mathcal{K}))=0$
of Conjecture \ref{con:stabilization}.
\begin{prop}
\label{pro:stabilize2}If $\mathcal{K}$ is any Legendrian knot, then
$\ell(S_{+}S_{-}(\mathcal{K}))=0$.\end{prop}
\begin{proof}
We will construct a closure $\bar{Y}$ of $\mathcal{K}'=S_{+}S_{-}(\mathcal{K})$
with an overtwisted disk, so that the vanishing follows from Corollary
B of \cite{Mrowka-knots}. Stabilization of a Legendrian knot $\mathcal{K}$
corresponds to attaching a bypass to its complement: if we stabilize
to get $S_{\pm}(\mathcal{K})$ inside a standard neighborhood $N(\mathcal{K})\subset Y$
and fix a standard neighborhood $N(\mathcal{K}^{\pm})\subset N(\mathcal{K})$,
then $Y\backslash N(\mathcal{K}^{\pm})$ is obtained from $Y\backslash N(\mathcal{K})$
by a bypass attachment. See \cite{Etnyre-torus} or the proof of Theorem
1.5 in \cite{Stipsicz-Vertesi} for discussion.

\begin{figure}
\begin{centering}
\includegraphics[scale=0.9]{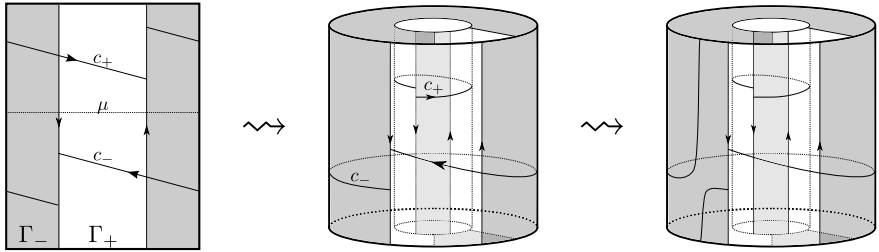}
\par\end{centering}

\caption{\label{fig:stabilize2}Attaching curves for bypasses in the complement
of $S_{+}S_{-}(\mathcal{K})$ and its closure.}

\end{figure}
In the leftmost part of Figure \ref{fig:stabilize2} we have indicated
the attaching arcs $c_{+}$ and $c_{-}$ of bypasses corresponding
to positive and negative stabilizations on $\partial(Y\backslash N(\mathcal{K}'))$,
as in Figure 10 of \cite{Stipsicz-Vertesi}, with a single meridional
suture $\mu$ between them. The dividing curves are shown with orientation
for convenience, so that they have the same orientation as the boundary
$\partial\Gamma_{+}$ of the positive region. If we start to form
the closure of $Y\backslash N(\mathcal{K}')$ by attaching a surface
$T\times I$ to neighborhoods of the sutures and then rounding edges,
we may then cut out $T\times I$ to get a contact manifold with corners
as in the middle figure; this indicates the positions of the arcs
$c_{\pm}$ on the boundary components $\bar{R}_{\pm}\subset Y'$.

We wish to glue $\bar{R}_{+}$ to $\bar{R}_{-}$ so that the arcs
$c_{+}$ and $c_{-}$ are glued together, but as shown in the middle
of Figure \ref{fig:stabilize2} we cannot do this by identifying the
inside and outside regions in the obvious way. Indeed, we must identify
the white component $(\bar{R}_{+})_{+}$ on the outside with the gray
component $(\bar{R}_{-})_{-}$ on the inside, identifying the left
dividing curve on the outside with the left dividing curve on the
inside and likewise for the right dividing curves, but then $c_{+}$
and $c_{-}$ cannot be made parallel so that they end up identified.
The problem is that as we follow them leftward and around the back
of the cylinder from the leftmost dividing curves, the arc $c_{-}$
ends {}``above'' its starting point whereas $c_{+}$ ends {}``below''
its starting point. However, we can glue $c_{+}$ to $c_{-}$ by applying
a Dehn twist to the outer gray annulus $(\bar{R}_{+})_{-}$ along
its core as shown on the right side of Figure \ref{fig:stabilize2}.
We can then {}``untwist'' $c_{-}$ by reparametrizing $(\bar{R}_{+})_{-}$,
sliding the lower endpoint of $c_{-}$ downward along its dividing
curve until it has nearly traversed the entire curve and lies just
above the other endpoint; this allows us to identify it with $(\bar{R}_{-})_{+}$
so that $c_{-}$ is sent to $c_{+}$. Now we can glue $\bar{R}_{+}$
to $\bar{R}_{-}$ so that $c_{-}$ and $c_{+}$ are identified, and
the union of their respective bypasses is an overtwisted disk in the
closure $(\bar{Y},\bar{\xi})$. We conclude that $\psi(\bar{Y},\bar{\xi})$,
and hence $\ell(\mathcal{K}')$, is zero. 
\end{proof}

\section{Contact surgery\label{sec:contact-surgery}}

\subsection{Behavior under contact $(+1)$-surgery}

The following is a direct analog of Theorem 1.1 of \cite{LOSS-surgeries},
which concerns the LOSS invariant $\hat{\mathcal{L}}(\mathcal{K})\in\widehat{HFK}(-Y,K)$
(or more generally $\mathcal{L}(\mathcal{K})\in HFK^{-}(-Y,K)$) but
is much harder to prove.
\begin{thm}
\label{thm:+1-surgery}Let $\mathcal{K}$ and $\mathcal{S}$ be disjoint
Legendrian knots in $(Y,\xi)$, and let $(Y_{\mathcal{S}},\xi_{\mathcal{S}})$
denote the contact manifold obtained by performing contact $(+1)$-surgery
along $\mathcal{S}$. Let $\mathcal{K}_{\mathcal{S}}$ be the image
of the Legendrian knot $\mathcal{K}$ in $Y_{\mathcal{S}}$. Then
there is a map $KHM(-Y,K)\otimes\mathcal{R}\to KHM(-Y_{\mathcal{S}},K_{\mathcal{S}})\otimes\mathcal{R}$
such that $\ell(\mathcal{K})\mapsto\ell(\mathcal{K}_{\mathcal{S}})$.\end{thm}
\begin{proof}
We may obtain $(Y,\xi)$ by performing contact $(-1)$-surgery on
$\mathcal{S}\subset Y_{\mathcal{S}}$ (see \cite[Proposition 8]{Ding-Geiges_fillability}).
Since $\mathcal{S}$ and $\mathcal{K}$ are disjoint it is easy to
see that we can fix a closure $\bar{Y}_{\mathcal{S}}$ of the complement
$Y_{\mathcal{S}}(\mathcal{K}_{\mathcal{S}})$ so that contact $(-1)$-surgery
on $\mathcal{S}\subset\bar{Y}_{\mathcal{S}}$ gives a closure $\bar{Y}$
of $Y(\mathcal{K})$, and the surface $\bar{R}$ and cycle $\eta\subset\bar{R}$
are the same in both closures. The Weinstein cobordism $(W,\omega)$
from $\bar{Y}_{\mathcal{S}}$ to $\bar{Y}$ coming from this handle
attachment gives a map \[
\Hto(W^{\dagger}):\Hto(-\bar{Y};\Gamma_{\eta})\to\Hto(-\bar{Y}_{\mathcal{S}};\Gamma_{\eta})\]
carrying $\ell(\mathcal{K})$ to $\ell(\mathcal{K}_{\mathcal{S}})$
by Corollary \ref{cor:surgery-cobordism}, and $\Hto(W^{\dagger},\mathfrak{s})(\ell(\mathcal{K}))$
is zero for all $\mathrm{Spin}^{c}$ structures $\mathfrak{s}\neq\mathfrak{s}_{\omega}$.
If we restrict $\Hto(W^{\dagger})$ to the $\mathrm{Spin}^{c}$ structures
on $W$ which are extremal with respect to $\bar{R}$ on each component
of the boundary, then we have a map \[
F_{W^{\dagger}}:KHM(-Y,K)\otimes\mathcal{R}\to KHM(-Y_{\mathcal{S}},K_{\mathcal{S}})\otimes\mathcal{R}\]
such that $F_{W^{\dagger},\mathfrak{s}}(\ell(\mathcal{K}))$ is $\ell(\mathcal{K}_{\mathcal{S}})$
for a unique $\mathrm{Spin}^{c}$ structure (again, $\mathfrak{s}_{\omega}$)
and zero for all others. 
\end{proof}
We can use Theorem \ref{thm:+1-surgery} to recover an analog of a
theorem of Sahamie \cite[Theorem 6.1]{Sahamie}.
\begin{thm}
\label{thm:+1-contact-invariant}Let $\mathcal{K}\subset(Y,\xi)$
be Legendrian, and let $(Y_{\pm},\xi_{\pm})$ be the result of a contact
$(\pm1)$-surgery along $\mathcal{K}$. These surgeries induce maps
\begin{eqnarray*}
KHM(-Y,K)\otimes\mathcal{R} & \to & \widetilde{HM}(-Y_{+})\otimes\mathcal{R}\\
\widetilde{HM}(-Y_{-})\otimes\mathcal{R} & \to & KHM(-Y,K)\otimes\mathcal{R}\end{eqnarray*}
sending $\ell(\mathcal{K})\mapsto\tilde{\psi}(\xi_{+})$ and $\tilde{\psi}(\xi_{-})\mapsto0$,
respectively.\end{thm}
\begin{proof}
Let $\mathcal{K}^{\pm}$ be Legendrian push-offs of $\mathcal{K}$
with an extra positive or negative twist around $\mathcal{K}$ as
in Figure \ref{fig:pushoff-twist}, so that $\mathcal{K}^{+}$ (resp.
$\mathcal{K}^{-}$) is Legendrian (resp. topologically) isotopic to
$\mathcal{K}$. As explained in the proof of Proposition 1 of \cite{DG-handle_moves},
performing a contact $(\pm1)$-surgery on $\mathcal{K}$ turns $\mathcal{K}^{\pm}$
into a meridian of the surgery torus with $tb=-1$, so that $\mathcal{K}^{\pm}$
becomes a Legendrian unknot in $(Y_{\pm},\xi_{\pm})$. In particular,
we have $KHM(-Y_{\pm},K^{\pm})\cong\widetilde{HM}(-Y_{\pm})$ and
$\ell(\mathcal{K}^{\pm}\subset Y_{\pm})=\tilde{\psi}(\xi_{\pm})$.

\begin{figure}
\begin{centering}
\includegraphics{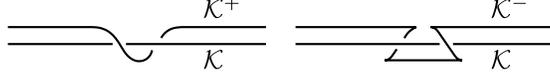}
\par\end{centering}

\caption{\label{fig:pushoff-twist}The knots $\mathcal{K}^{+}$ and $\mathcal{K}^{-}$
are constructed by adding a positive twist and a negative twist, respectively,
to a Legendrian push-off of $\mathcal{K}$.}

\end{figure}
Writing $\mathcal{S}=\mathcal{K}$ and applying Theorem \ref{thm:+1-surgery}
to $(Y,\xi)$, we now have a map \[
KHM(-Y,K^{+})\otimes\mathcal{R}\to KHM(-Y_{+},K^{+})\otimes\mathcal{R}\]
sending $\ell(\mathcal{K}^{+})=\ell(\mathcal{K})$ to $\ell(\mathcal{K}^{+}\subset Y_{+})=\tilde{\psi}(\xi_{+})$.
Similarly, if we let $\mathcal{S}\subset(Y_{-},\xi_{-})$ be the core
of the contact $(-1)$-surgery torus, then a contact $(+1)$-surgery
on $\mathcal{S}$ cancels the original $(-1)$-surgery, leaving the
original contact manifold $(Y,\xi)$. Theorem \ref{thm:+1-surgery}
then produces a map \[
KHM(-Y_{-},K^{-})\otimes\mathcal{R}\to KHM(-Y,K^{-})\otimes\mathcal{R}\]
which sends $\ell(\mathcal{K}^{-}\subset Y_{-})=\tilde{\psi}(\xi_{-})$
to $\ell(\mathcal{K}^{-}\subset Y)$. But $\mathcal{K}^{-}$ is Legendrian
isotopic in $Y$ to the double stabilization $S_{+}S_{-}(\mathcal{K})$,
hence $\ell(\mathcal{K}^{-}\subset Y)=0$ by Proposition \ref{pro:stabilize2}
and we are done.\end{proof}
\begin{cor}
If the result of contact $(+1)$-surgery on $\mathcal{K}\subset(Y,\xi)$
has nonzero contact invariant $\psi(\xi_{+})$, then $\ell(\mathcal{K})\not=0$.\end{cor}
\begin{proof}
Proposition \ref{pro:general-unknot} provides a map $\widetilde{HM}(-Y)\otimes\mathcal{R}\to\Hto(-Y)\otimes\mathcal{R}$
sending $\tilde{\psi}(\xi_{+})$ to $\psi(\xi_{+})\otimes1$, so if
$\psi(\xi_{+})\not=0$ then $\tilde{\psi}(\xi_{+})\not=0$ and hence
$\ell(\mathcal{K})\not=0$ as well.
\end{proof}
For example, let $K\subset S^{3}$ be a knot with smooth slice genus
$g_{s}\geq1$, and suppose we have a Legendrian representative $\mathcal{K}\subset(S^{3},\xi_{\mathrm{std}})$
of $K$ with $tb(\mathcal{K})=2g_{s}-1$. Let $(Y_{+},\xi_{+})$ denote
the result of contact $(+1)$-surgery on $\mathcal{K}$. The following
argument of Lisca and Stipsicz \cite{Lisca-Stipsicz_tight}, translated
directly from Heegaard Floer to monopole Floer homology, shows that
$\psi(\xi_{+})\not=0$.

Letting $W$ denote the Weinstein cobordism from $(Y_{+},\xi_{+})$
to $(S^{3},\xi_{std})$ which reverses the contact $(+1)$-surgery
along $\mathcal{K}$ , we have a map \[
\Hto(-S^{3})\stackrel{F_{W^{\dagger}}}{\longrightarrow}\Hto(-Y_{+})\]
sending $\psi(\xi_{\mathrm{std}})\not=0$ to $\psi(\xi_{+})$ by Corollary
\ref{cor:surgery-cobordism}, so we wish to show that $F_{W^{\dagger}}$
is injective. Now $Y_{+}$ is the result of a topological $2g_{s}$-surgery
on $K$, so $-Y_{+}$ is the result of a $-2g_{s}$-surgery on the
mirror image $\bar{K}$ and thus $F_{W^{\dagger}}$ fits into a surgery
exact triangle \[
\xymatrix{\Hto(S^{3})\ar[rr]^{F_{W^{\dagger}}} &  & \Hto(S_{-2g_{s}}^{3}(\bar{K}))\ar[ld]\\
 & \Hto(S_{-2g_{s}+1}^{3}(\bar{K}))\ar[lu]^{F_{V}}}
\]
where $V$ is a 2-handle cobordism. Lisca and Stipsicz show that $V$
contains a closed surface $\Sigma$ of genus $g_{s}>0$ and self-intersection
$2g_{s}-1\geq0$. If $\mathfrak{s}$ is a $\mathrm{Spin}^{c}$-structure
for which $F_{V,\mathfrak{s}}\not=0$, then the adjunction inequality
for cobordisms says that \[
|\langle c_{1}(\mathfrak{s}),\Sigma\rangle|+\Sigma\cdot\Sigma\leq2g(\Sigma)-2,\]
hence $|\langle c_{1}(\mathfrak{s}),\Sigma\rangle|\leq-1$, a contradiction.
Therefore $F_{V}$ is zero and $F_{W^{\dagger}}$ is injective by
exactness.
\begin{cor}
\label{cor:nonzero-slice-bennequin}If $\mathcal{K}\subset(S^{3},\xi_{\mathrm{std}})$
is a Legendrian representative of a knot $K$ with slice genus $g_{s}>0$
and $tb(\mathcal{K})=2g_{s}-1$, then $\ell(\mathcal{K})\not=0$.
Examples include any topologically nontrivial $\mathcal{K}$ for which
$tb(\mathcal{K})=2g(K)-1$, where $g(K)$ is the Seifert genus of
$K$.
\end{cor}
For example, in \cite{Lisca-Stipsicz_tight} the authors remark that
$\overline{tb}(K)=2g(K)-1$ for any algebraic knot, where $\overline{tb}$
denotes maximal Thurston-Bennequin number. More generally, if $\mathcal{K}$
is the Legendrian closure of a positive braid \cite{Kalman-families}
with $n$ strands and $c$ crossings then it is easy to compute that
$tb(\mathcal{K})=c-n=2g(K)-1$, hence closures of positive braids
have the same property.

For any Legendrian knot $\mathcal{K}$, the Legendrian Whitehead double
$W(\mathcal{K})$ (due to Eliashberg, and denoted $\Gamma_{\mathrm{dbl}}$
by Fuchs in \cite{Fuchs-augmentations}) is constructed by taking
$\mathcal{K}$ and a slight push-off $\mathcal{K}'$ in the $z$-direction,
and then replacing a pair of parallel segments with a clasp as in
Figure \ref{fig:double-clasp}; it has genus 1 and $tb=1$.

Thus $\ell(\mathcal{K})\not=0$ if $\mathcal{K}$ is a $tb$-maximizing
representative of the closure of a positive braid or a Legendrian
Whitehead double. Similarly, there are many examples of knots with
$tb(\mathcal{K})=2g_{s}(K)-1$ where $1\le g_{s}(K)<g(K)$, and these
all have $\ell(\mathcal{K})\not=0$; according to KnotInfo \cite{KnotInfo},
the smallest examples have topological types $m(8_{21})$ and $m(9_{45})$.

\begin{figure}
\begin{centering}
\includegraphics{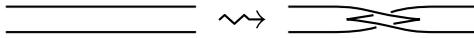}
\par\end{centering}

\caption{\label{fig:double-clasp}Constructing a Legendrian Whitehead double
from $\mathcal{K}$ and its push-off.}

\end{figure}

\subsection{Non-loose knots}

By Proposition \ref{pro:overtwisted-torsion}, a Legendrian knot $\mathcal{K}$
in an overtwisted manifold is non-loose if $\ell(\mathcal{K})\not=0$.
Our goal in this section is to apply Theorem \ref{thm:+1-surgery}
to construct examples where this is the case. In order to do so, we
will first need the following lemma on monopole knot homology and
surgery.
\begin{lem}
\label{lem:relate-surgery-maps}Let $K,S\subset Y$ be disjoint knots,
and for any integral framing $f$, let $K_{f}$ denote the image of
$K$ in the manifold $Y_{f}$ obtained by $f$-surgery along $S$.
For each $f$ there is a map $s_{f}:KHM(Y,K)\to KHM(Y_{f},K_{f})$
corresponding to a 2-handle attachment along $S$ in a closure $\bar{Y}$
of $Y\backslash K$, and these maps satisfy the following:
\begin{enumerate}
\item If $s_{f+1}$ is either injective or surjective, then $s_{f}$ is
injective.
\item If $s_{f}$ is either surjective or zero, then $s_{f+1}$ is zero.
\end{enumerate}
\end{lem}
\begin{proof}
We have a surgery exact triangle \[
\xymatrix{\Hto(\bar{Y}|\bar{R})\ar[rr]^{F_{1}} &  & \Hto(\bar{Y}_{f}|\bar{R})\ar[ld]^{F_{2}}\\
 & \Hto(\bar{Y}_{f+1}|\bar{R})\ar[lu]^{F_{3}}}
\]
where $\bar{Y}$, $\bar{Y}_{f}$, and $\bar{Y}_{f+1}$ are closures
of the complements of $K$ in $Y$, $Y_{f},$ and $Y_{f+1}$; note
that $\Hto(\bar{Y}_{f}|\bar{R})=KHM(Y_{f},K_{f})$ by definition.
Similarly, we have a second triangle of the form \[
\xymatrix{\Hto(\bar{Y}|\bar{R})\ar[rr]^{G_{1}} &  & \Hto(\bar{Y}_{f+1}|\bar{R})\ar[ld]^{G_{2}}\\
 & \Hto(\bar{Y}_{f+2}|\bar{R})\ar[lu]^{G_{3}}}
\]
and by \cite[Proposition 7.2]{KMOS-lens} we have $G_{1}\circ F_{3}=F_{3}\circ G_{1}=0$,
since each composition comes from a cobordism created by a pair of
$2$-handles which contains a homologically nontrivial sphere of self-intersection
zero. This implies that $G_{1}$ (resp. $F_{3}$) is zero if $F_{3}$
(resp. $G_{1}$) is either injective or surjective.

Suppose $G_{1}$ is either injective or surjective; then $F_{3}=0$,
hence by exactness $F_{1}$ is injective. Similarly, if $F_{1}$ is
surjective then $F_{2}$ is zero, hence $F_{3}$ is injective, and
if $F_{1}$ is zero then $F_{3}$ is surjective; either of these imply
$G_{1}=0$. Since $F_{1}=s_{f}$ and $G_{1}=s_{f+1}$, we are done.\end{proof}
\begin{prop}
\label{pro:+1-surgery-nonvanishing}Let $\mathcal{K},\mathcal{S}\subset(Y,\xi)$
be nullhomologous Legendrian knots such that $S$ is homotopic to
a meridian of $K$ in $Y\backslash K$. If $\mathcal{K}_{\mathcal{S}}\subset(Y_{\mathcal{S}},\xi_{\mathcal{S}})$
is the image of $\mathcal{K}$ in the manifold obtained by contact
$(+1)$-surgery on $\mathcal{S}$, then $\ell(\mathcal{K}_{\mathcal{S}})\not=0$
if and only if $\ell(\mathcal{K})\not=0$ and $tb(\mathcal{S})\geq0$.\end{prop}
\begin{proof}
Observe that $Y_{\mathcal{S}}$ is obtained from $Y$ by a topological
$(tb(\mathcal{S})+1)$-surgery along $S$, so $-Y_{\mathcal{S}}$
is related to $-Y$ by a $k$-surgery along $S$, where $k=-tb(\mathcal{S})-1$.
In the notation of Lemma \ref{lem:relate-surgery-maps}, the map \[
s_{k}:KHM(-Y,K)\otimes\mathcal{R}\to KHM(-Y_{\mathcal{S}},K_{\mathcal{S}})\otimes\mathcal{R}\]
 carries $\ell(\mathcal{K})$ to $\ell(\mathcal{K}_{\mathcal{S}})$
as in Theorem \ref{thm:+1-surgery}, so we will show that $s_{k}$
is injective if $k\leq-1$ (i.e. if $tb(\mathcal{S})\geq0$) and zero
if $k\geq0$. By Lemma \ref{lem:relate-surgery-maps} it will be enough
to show that $KHM((-Y)_{0},K_{0})=\Hto((-\bar{Y})_{0}|\bar{R})=0$,
where $(\bar{Y},\bar{R})$ is a closure of $Y\backslash\mathcal{K}$
and $(-\bar{Y})_{0}$ is obtained by $0$-surgery on $S\subset-\bar{Y}$.
Indeed, this implies that $s_{k}$ is zero for $k=0$ and hence for
all $k\geq0$, and since $s_{0}$ is also surjective it follows that
$s_{k}$ is injective for all $k\leq-1$.

Since $S$ and a meridian of $K$ are homotopic in $Y\backslash K$,
they are homotopic in $-\bar{Y}$ as well, and in particular $S$
is homotopic to a nonseparating curve $c\subset\bar{R}$. When we
perform $0$-surgery along $S$ to obtain $-\bar{Y}_{0}$, then, the
curve $c$ becomes nullhomotopic and so $[\bar{R}]\in H_{2}(-\bar{Y}_{0})$
has a representative of genus $g(\bar{R})-1$. Since $g(\bar{R})\geq2$,
the adjunction inequality tells us that $\Hto(-\bar{Y}_{0}|\bar{R})=0$
as desired.\end{proof}
\begin{cor}
\label{cor:nonloose-knots}Let $\mathcal{K}\cup\mathcal{S}$ be a
two-component Legendrian link in $(S^{3},\xi_{\mathrm{std}})$ satisfying
the following:
\begin{enumerate}
\item $\mathcal{K}$ is a Legendrian unknot with $tb(\mathcal{K})=-1$.
\item $tb(\mathcal{S})\geq0$.
\item The linking number $lk(K,S)$ is $\pm1$.
\end{enumerate}
Then $\mathcal{K}_{\mathcal{S}}$ is a non-loose Legendrian knot in
the contact manifold $(S_{\mathcal{S}}^{3},\xi_{\mathcal{S}})$, where
the subscript denotes contact $(+1)$-surgery along $\mathcal{S}$.\end{cor}
\begin{proof}
We know that $\ell(\mathcal{K})=1\in\mathcal{R}$ by Proposition \ref{pro:unknot-invariant},
so Proposition \ref{pro:+1-surgery-nonvanishing} tells us that $\ell(\mathcal{K}_{\mathcal{S}})\not=0$.
\end{proof}
In particular, given a knot $\mathcal{S}$ with $tb(\mathcal{S})>0$
which satisfies all the other conditions of Corollary \ref{cor:nonloose-knots},
we can stabilize $\mathcal{S}$ to get $\mathcal{S}'$ with $tb(\mathcal{S}')\geq0$
and then apply the corollary to $\mathcal{K}\cup\mathcal{S}'$. Since
$\mathcal{S}'$ is stabilized, the contact $(+1)$-surgery results
in an overtwisted contact structure (see for example \cite{Ding-Geiges-Stipsicz}),
and $\mathcal{K}_{\mathcal{S}'}$ is non-loose.

For example, the right handed trefoil has a unique Legendrian representative
with $(tb,r)=(1,0)$. Let $\mathcal{S}$ be a stabilization of this
Legendrian knot, and let $\mathcal{K}$ be a Legendrian unknot with
$lk(\mathcal{K},\mathcal{S})=\pm1$. Then $(+1)$-surgery on $\mathcal{S}$
gives an overtwisted contact structure on the Poincaré homology sphere
$-P=-\Sigma(2,3,5)$, which in fact does not admit tight positive
contact structures \cite{Etnyre-nonexistence}, and the image $\mathcal{K}_{\mathcal{S}}$
of $\mathcal{K}$ in $-P$ is a non-loose knot. We exhibit a family
$\mathcal{K}_{n}$ of such knots in Figure \ref{fig:surgery-knots}.

\begin{figure}
\begin{centering}
\includegraphics{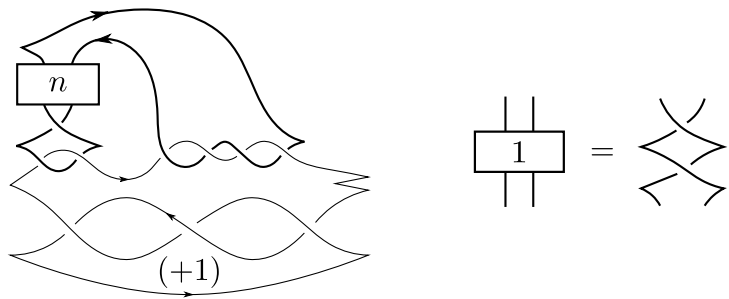}
\par\end{centering}

\caption{\label{fig:surgery-knots}A Legendrian knot $\mathcal{K}_{n}$ in
the overtwisted contact structure on $-P$ obtained from $(S^{3},\xi_{\mathrm{std}})$
by contact $(+1)$-surgery on a stabilized right-handed trefoil $\mathcal{T}$.}
\end{figure}

\begin{prop}
The knots $K_{n}$ ($n\geq0$) are all distinct, and none of them
are fibered.\end{prop}
\begin{proof}
Let $L=U_{n}\cup T\subset S^{3}$ and $\hat{L}=K_{n}\cup\hat{T}\subset-P$,
where $\hat{T}$ is the core of the surgery torus glued to $S^{3}\backslash T$
to obtain $-P$. We will compute the Conway polynomial $\nabla_{L}$
and use it to determine $\nabla_{\hat{L}}$ and $\nabla_{K_{n}}$,
and hence the Alexander polynomial $\Delta_{K_{n}}$, referring to
the results of \cite{BoyerLines-Conway}; note that $\nabla$ and
$\Delta$ are related by \[
\nabla_{L}(s_{1},\dots,s_{m})=\begin{cases}
(s_{1}-s_{1}^{-1})^{-1}\Delta_{L}(s_{1}^{2}), & |L|=1\\
\Delta_{L}(s_{1}^{2},s_{2}^{2},\dots,s_{m}^{2}), & |L|>1.\end{cases}\]

Since $\hat{L}$ is obtained as the cores of surgery tori for $\frac{1}{0}$-surgery
on $U_{n}$ and $\frac{1}{1}$-surgery on $T$, and $lk(U_{n},T)=-1$,
the link $\hat{L}$ is determined by $L$ and the framing matrix \[
B=\left(\begin{matrix}1 & -1\\
0 & 1\end{matrix}\right),\]
from which we can determine $lk_{-P}(K_{n},\hat{T})=-1$ and \[
\nabla_{\hat{L}}(s_{1},s_{2})=\nabla_{L}(s_{1},s_{1}s_{2})\]
by the {}``variance under surgery'' proposition. Then $\nabla_{\hat{L}}(s,1)=-(s-s^{-1})\nabla_{K_{n}}(s)$
by {}``restriction,'' and so \[
\Delta_{K_{n}}(s^{2})=(s-s^{-1})\nabla_{K_{n}}(s)=-\nabla_{\hat{L}}(s,1)=-\nabla_{L}(s,s).\]
Then we have reduced the computation of $\Delta_{K_{n}\subset-P}$
to that of $\nabla_{L\subset S^{3}}(s,s)$. Note that the latter term
is determined entirely by a skein relation $\nabla_{L_{+}}-\nabla_{L_{-}}=(s-s^{-1})\nabla_{L_{0}}$:

\begin{center}
\includegraphics{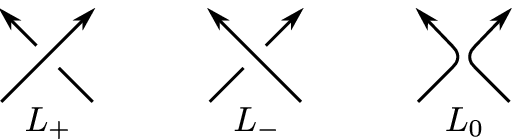}
\par\end{center}

\noindent and by $\nabla_{U}(s)=(s-s^{-1})^{-1}$, where $U$ is the
unknot.

Using the skein relation at a crossing in one of the $n$ full twists
of $U_{n}$, we see that \[
\nabla_{U_{n}\cup T}-\nabla_{U_{n-1}\cup T}=(s-s^{-1})\nabla_{L_{0}}\]
and so $\nabla_{L}=\nabla_{U_{0}\cup T}+n(s-s^{-1})\nabla_{L_{0}}$.
Applying the skein relation to the crossing of $U_{n}$ directly below
the $n$ twists, when $n=0$, we get $\nabla_{U_{0}\cup T}-\nabla_{L_{1}}=(s-s^{-1})\nabla_{L_{0}}$,
hence \[
\nabla_{L}=\nabla_{L_{1}}+(n+1)(s-s^{-1})\nabla_{L_{0}}\]
where $L=U_{n}\cup T$, $L_{0}$, and $L_{1}$ are the links in Figure
\ref{fig:skein-links}.

\begin{figure}
\begin{centering}
\includegraphics[scale=0.9]{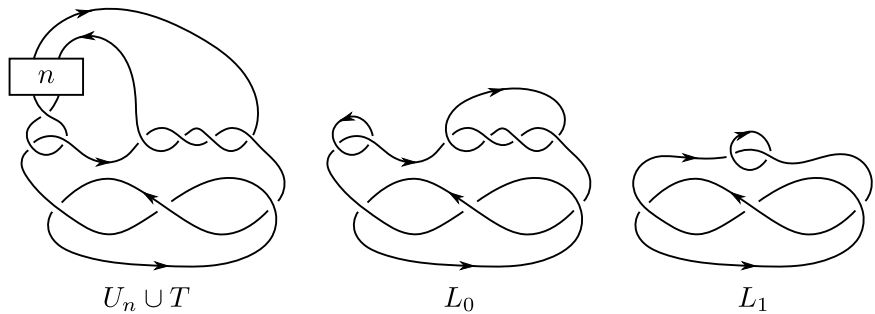}
\par\end{centering}

\caption{\label{fig:skein-links}Links appearing in the computation of $L=U_{n}\cup T$
by the skein relation.}
\end{figure}
A straightforward computation now yields \begin{eqnarray*}
\nabla_{L_{0}}(s,s,s) & = & -2(s-s^{-1})(s^{2}-1+s^{-2})\\
\nabla_{L_{1}}(s,s) & = & -(s^{2}-1+s^{-2})\end{eqnarray*}
and since $\Delta_{K_{n}}(s^{2})=-\nabla_{L_{1}}-(n+1)(s-s^{-1})\nabla_{L_{0}}$
we conclude that \[
\Delta_{K_{n}}(t)=(t-1+t^{-1})(1+2(n+1)(t-2+t^{-1})).\]
Since the Alexander polynomials $\Delta_{K_{n}}(t)$ are all distinct,
so are the $K_{n}$; and since $\Delta_{K_{n}}$ is never monic, the
$K_{n}$ cannot be fibered.
\end{proof}
We remark that in general very few examples of non-loose knots in
overtwisted contact manifolds have been studied. Etnyre \cite{Etnyre-contact_surgery}
observed that if the result $(S_{\mathcal{K}}^{3},\xi_{\mathcal{K}})$
of contact $(+1)$-surgery on $\mathcal{K}\subset(S^{3},\xi_{\mathrm{std}})$
is overtwisted, then the core $\mathcal{K}'$ of the surgery torus
in $S_{\mathcal{K}}^{3}$ is non-loose. (If $S_{\mathcal{K}}^{3}$
is the Poincaré homology sphere with either orientation, then $K$
must be a trefoil \cite{Ghiggini} and so one can show that $\Delta_{K'}=\pm\Delta_{K}\not=\Delta_{K_{n}}$.)
Furthermore, Etnyre and Vela-Vick \cite{Etnyre-Vick} showed that
given an open book decomposition which supports $(Y,\xi)$, any Legendrian
approximation of the binding is non-loose. To the best of our knowledge,
these are the only known examples in manifolds other than $S^{3}$.

In particular, it seems that the non-loose knots $\mathcal{K}_{n}\subset-P$
were not previously known, and in fact may be the only known non-fibered
examples (even in $S^{3}$) which are not the cores of surgery tori.
The construction of Corollary \ref{cor:nonloose-knots} is of course
much more general; it would be interesting to give examples of links
$\mathcal{K}_{i}\cup\mathcal{S}$ ($i=1,2$) which are topologically
but not Legendrian isotopic and which give distinct non-loose knots
$(\mathcal{K}_{i})_{\mathcal{S}}$.

\section{Lagrangian concordance\label{sec:Lagrangian-concordance}}

Chantraine \cite{Chantraine} defined an interesting notion of concordance
on the set of all Legendrian knots in a contact $3$-manifold $Y$.
\begin{defn}
Let $\mathcal{K}_{0}$ and $\mathcal{K}_{1}$ be Legendrian knots
parametrized by embeddings $\gamma_{i}:S^{1}\to Y$, and let $Y\times\mathbb{R}$
be the symplectization of $Y$. We say that $\mathcal{K}_{0}$ is
\emph{Lagrangian concordant} to $\mathcal{K}_{1}$, denoted $\mathcal{K}_{0}\prec\mathcal{K}_{1}$,
if there is a Lagrangian embedding $L:S^{1}\times\mathbb{R}\hookrightarrow Y\times\mathbb{R}$
and a $T>0$ such that $L(s,t)=(\gamma_{0}(s),t)$ for $t<-T$ and
$L(s,t)=(\gamma_{1}(s),t)$ for $t>T$.\end{defn}
\begin{thm}[\cite{Chantraine}]
The relation $\prec$ descends to a relation on Legendrian isotopy
classes of Legendrian knots. If $\mathcal{K}_{0}\prec\mathcal{K}_{1}$
then $tb(\mathcal{K}_{0})=tb(\mathcal{K}_{1})$ and $r(\mathcal{K}_{0})=r(\mathcal{K}_{1})$.
\end{thm}
Our goal in this section is to investigate the behavior of $\ell(\mathcal{K})$
under Lagrangian concordance:
\begin{thm}
\label{thm:concordance-map}Let $\mathcal{K}_{0},\mathcal{K}_{1}$
be Legendrian knots in a contact homology $3$-sphere $Y$ satisfying
$\mathcal{K}_{0}\prec\mathcal{K}_{1}$. Then there is a homomorphism
\[
KHM(-Y,K_{1})\otimes\mathcal{R}\to KHM(-Y,K_{0})\otimes\mathcal{R}\]
sending $\ell(\mathcal{K}_{1})$ to $\ell(\mathcal{K}_{0})$.
\end{thm}
We compare this with the remarks in \cite[Section 5.2]{Chantraine},
where it is observed that Lagrangian concordance induces a map $LCH(\mathcal{K}_{1})\to LCH(\mathcal{K}_{0})$
on Legendrian contact homology.
\begin{proof}
We fix a particular closure $(\bar{Y}_{i},\bar{R}_{i})$ of the sutured
knot complements $Y(\mathcal{K}_{i})$: place the meridional sutures
close together so that in $\partial(Y\backslash\mathcal{K}_{i})$
they bound an annulus $A$ in which the dividing curves are parallel
to a longitude. In the other annulus $A'$ bounded by the sutures,
the dividing curves twist around the meridional direction a total
of $tb(\mathcal{K}_{i})$ times; recall that $tb(\mathcal{K}_{0})=tb(\mathcal{K}_{1})$.
We glue a surface $T\times I$ to each complement and round edges,
resulting in a manifold with boundary $\bar{R}_{+}\sqcup\bar{R}_{-}$
and $\mathrm{int}(A)\subset\bar{R}_{+}$. Finally, we glue $\bar{R}_{+}$
to $\bar{R}_{-}$ by identifying $(x,1)\in T\times\{1\}$ to $(x,-1)\in T\times\{-1\}$
for all $x\in\mathrm{int}(T)$, and identifying $A$ with $A'$ by
a homeomorphism composed of enough Dehn twists around the core of
$A$ to make the dividing curves match.

This construction guarantees that $Z_{0}=\bar{Y}_{0}\backslash\mathrm{int}(Y\backslash\mathcal{K}_{0})$
and $Z_{1}=\bar{Y}_{1}\backslash\mathrm{int}(Y\backslash\mathcal{K}_{1})$
are contactomorphic as $3$-manifolds with torus boundary. In the
symplectization $Y\times\mathbb{R}$, the cylinder $\mathcal{K}_{0}\times\mathbb{R}$
is Lagrangian, hence it has a standard neighborhood symplectomorphic
to a neighborhood $N$ of the $0$-section in $T^{*}(S^{1}\times\mathbb{R})$.
Then a neighborhood of the boundary $T^{2}\times\mathbb{R}$ of the
symplectization $Z_{0}\times\mathbb{R}$, can be identified with the
complement of the $0$-section in $N$.

Now consider the Lagrangian cylinder $\mathcal{L}\subset Y\times\mathbb{R}$
defining the concordance from $\mathcal{K}_{0}$ to $\mathcal{K}_{1}$.
Once again, $\mathcal{L}$ has a neighborhood symplectomorphic to
$N$; if we remove a sufficiently small neighborhood of $\mathcal{L}$,
then there is a collar neighborhood of $\partial((Y\times\mathbb{R})\backslash\mathcal{L})$
which is orientation-reversing symplectomorphic to $N$ with the $0$-section
removed. Thus we can glue $(Y\times\mathbb{R})\backslash\mathcal{L}$
to $Z_{0}\times\mathbb{R}$ to get a symplectic manifold $W$ with
two infinite ends. One of these ends is a piece $\bar{Y}_{0}\times(-\infty,T]$
of the symplectization of $\bar{Y}_{0}$, and since $Z_{0}$ is contactomorphic
to $Z_{1}$ the other end is $\bar{Y}_{1}\times[T,\infty)$. Thus
$W$ is a boundary-exact symplectic cobordism from $\bar{Y}_{0}$
to $\bar{Y}_{1}$.

Finally, we wish to show that the map $i^{*}:H^{1}(W,\bar{Y}_{1})\to H^{1}(\bar{Y}_{0})$
is zero. By Poincaré duality it suffices to show that $H_{3}(W,\bar{Y}_{0})\to H_{2}(\bar{Y}_{0})$
is zero, or equivalently (by the long exact sequence of the pair $(W,\bar{Y}_{0})$)
that the map $H_{2}(\bar{Y}_{0})\to H_{2}(W)$ is injective. But there
is a natural isomorphism $H_{2}((Y\times\mathbb{R})\backslash\mathcal{L})\cong H_{2}(Y\backslash\mathcal{K}_{0})$
by Alexander duality, hence by the Mayer-Vietoris sequence and the
five lemma it follows that $H_{2}(\bar{Y}_{0})\to H_{2}(W)$ is an
isomorphism as well, and so $i^{*}$ is indeed zero.

Since $W$ is a boundary-exact symplectic cobordism and $H^{1}(W,\bar{Y}_{1})\to H^{1}(\bar{Y}_{0})$
is zero, we apply Theorem \ref{thm:symplectic-cobordism} to conclude
that \[
\psi(\bar{Y}_{0},\bar{\xi}_{0})=\Hto(W^{\dagger},\mathfrak{s}_{\omega})\psi(\bar{Y}_{1},\bar{\xi}_{1}).\]
Thus $\Hto(W^{\dagger},\mathfrak{s}_{\omega})$ induces a map $f:KHM(-Y,K_{1})\otimes\mathcal{R}\to KHM(-Y,K_{0})\otimes\mathcal{R}$
satisfying $f(\ell(\mathcal{K}_{1}))=\ell(\mathcal{K}_{0})$, as desired.\end{proof}
\begin{cor}
If $\mathcal{K}_{0}\prec\mathcal{K}_{1}$ and $\ell(\mathcal{K}_{0})$
is nonzero, then so is $\ell(\mathcal{K}_{1})$.
\end{cor}

\begin{cor}
If a Legendrian knot $\mathcal{K}\subset(S^{3},\xi_{0})$ bounds a
Lagrangian disk in the standard symplectic 4-ball $B^{4}$, then $\ell(\mathcal{K})$
is a unit of $KHM(-S^{3},K)$.\end{cor}
\begin{proof}
The Legendrian unknot $\mathcal{U}$ is Lagrangian concordant to $\mathcal{K}$,
and $\ell(\mathcal{U})$ is a generator of $KHM(-S^{3},U)\otimes\mathcal{R}\cong\mathcal{R}$
by Proposition \ref{pro:unknot-invariant}, so by Theorem \ref{thm:concordance-map}
there is a map $KHM(-S^{3},K)\otimes\mathcal{R}\to\mathcal{R}$ such
that the image of $\ell(\mathcal{K})$ is a unit.
\end{proof}
It is observed in the addendum to \cite{Chantraine} that the following
tangle replacement in the front projection (obtained from a 1-smoothing
of a crossing in the Lagrangian projection) can be realized by a Lagrangian
saddle cobordism:

\begin{center}
\includegraphics{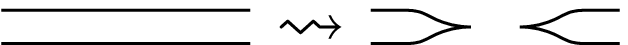}
\par\end{center}

\noindent If such a move turns a Legendrian knot $\mathcal{K}$ into
a Legendrian unlink whose components are both $\mathcal{U}$, we can
cap both components with Lagrangian disks and thus build a Lagrangian
slice disk for $\mathcal{K}$, proving that $\ell(\mathcal{K})$ is
a primitive element of $KHM(-S^{3},K)$. Figure \ref{fig:Lagrangian-slice}
shows grid diagrams for seven such knots, of topological types $m(9_{46})$,
$m(10_{140})$, $m(10_{140})$, $11n_{139}$, $m(12n_{582})$, $m(12n_{768})$,
and $m(12n_{838})$, which were discovered using a combination of
KnotInfo \cite{KnotInfo}, the Legendrian knot atlas \cite{Chongchitmate-Ng},
and Gridlink \cite{Gridlink}. As usual, these may be turned into
front projections of Legendrian knots by smoothing out all northeast
and southwest corners and then rotating 45 degrees counterclockwise.
The dotted line in each diagram indicates where to perform the tangle
replacement.

\begin{figure}
\begin{centering}
\includegraphics[scale=0.75]{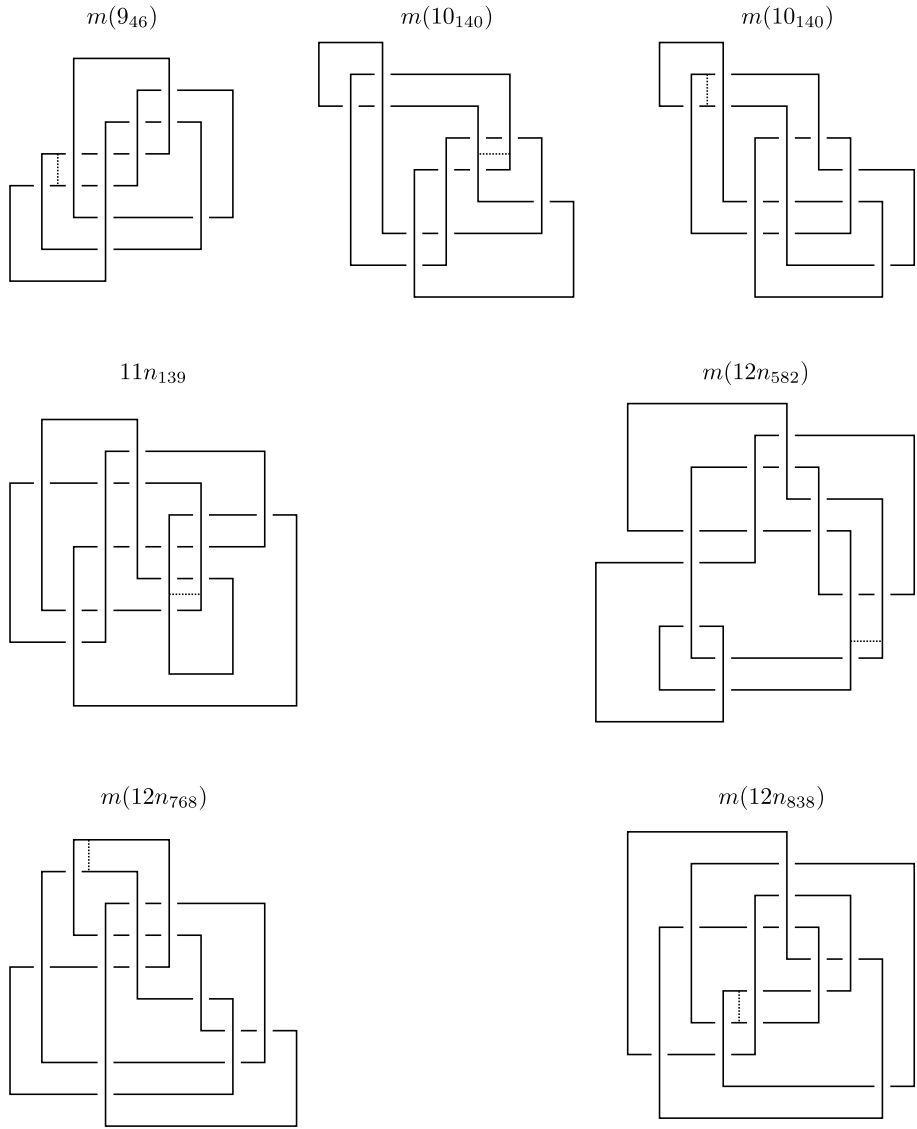}
\par\end{centering}

\caption{Seven Lagrangian knots $\mathcal{K}$ which bound Lagrangian disks
in $B^{4}$ and thus satisfy $\ell(\mathcal{K})\not=0$.\label{fig:Lagrangian-slice}}

\end{figure}

\begin{conjecture}
Given a Lagrangian cobordism $\mathcal{K}_{0}\prec_{\Sigma}\mathcal{K}_{1}$
of arbitrary genus, there is a map $KHM(-Y,K_{1})\otimes\mathcal{R}\to KHM(-Y,K_{0})\otimes\mathcal{R}$
sending $\ell(\mathcal{K}_{1})$ to $\ell(\mathcal{K}_{0})$.
\end{conjecture}
\bibliographystyle{amsplain}
\bibliography{references}

\end{document}